\documentclass[12pt, reqno]{amsart}
\UseRawInputEncoding
\usepackage{amsmath, amsthm, amscd, amsfonts, amssymb, graphicx, color}
\usepackage[bookmarksnumbered, colorlinks, plainpages]{hyperref}
\input{mathrsfs.sty}

\hypersetup{colorlinks=true,linkcolor=red, anchorcolor=green,
citecolor=cyan, urlcolor=red, filecolor=magenta, pdftoolbar=true}

\newtheorem{theorem}{Theorem}[section]
\newtheorem{lemma}[theorem]{Lemma}

\newtheorem{corollary}[theorem]{Corollary}
\theoremstyle{definition}
\newtheorem{definition}[theorem]{Definition}

\theoremstyle{remark}
\newtheorem{remark}[theorem]{Remark}

\numberwithin{equation}{section}

\usepackage{color}

\begin{document}

\setcounter{page}{1}

\title[An extension of the $a$-numerical radius on $C^*$-algebras]
{An extension of the $a$-numerical radius on $C^*$-algebras}

\author[M.~Mabrouk \MakeLowercase{and} A.~Zamani]
{Mohamed Mabrouk$^{1}$ \MakeLowercase{and} Ali Zamani$^{2,*}$}

\address{$^1$Faculty of Sciences of Sfax, Department of Mathematics, University of Sfax, Tunisia}
\email{mbs\_mabrouk@yahoo.fr}

\address{$^*$Corresponding author, $^2$School of Mathematics and Computer Sciences, Damghan University, P.O.BOX 36715-364, Damghan, Iran}
\email{zamani.ali85@yahoo.com}

%%%%%%%%%%%%%%%%%%%%%%%%%%%%%%%%%%
%%%%%%%%%%%%%%%%%%%%%%%%%%%%%%%%%%
\begin{abstract}
Let $a$ be a positive element in a unital $C^*$-algebra $\mathfrak{A}$.
We define a semi-norm on $\mathfrak{A}$, which generalizes the $a$-operator semi-norm and the $a$-numerical radius.
We investigate basic properties of this semi-norm and prove inequalities involving it.
Further, we derive new upper and lower bounds for the $a$-numerical radii of elements in $\mathfrak{A}$.
Some other related results are also discussed.
\end{abstract}
\maketitle
%%%%%%%%%%%%%%%%%%%%%%%%%%%%%
\section{Introduction and preliminaries}
%%%%%%%%%%%%%%%%%%%%%%%%%%%%
Throughout this paper, let $\mathfrak{A}$ be a unital $C^*$-algebra with unit denoted by $\textbf{1}$
and we assume that $a\in\mathfrak{A}$ is a positive element.
Let $\mathbb{B}(\mathcal{H})$ be the $C^*$-algebra of all bounded linear operators on a complex Hilbert
space $\big(\mathcal{H}, \langle \cdot, \cdot\rangle\big)$.
Let $\mathcal{S}(\mathfrak{A})$ denote the set of states on $\mathfrak{A}$,
which is the set of all positive linear functionals $f$ on $\mathfrak{A}$ such that $\|f\|=1$.
By $V(x)$ we denote the (algebraic) numerical range of an element $x\in\mathfrak{A}$, that is, the set
$V(x)=\big\{f(x): f\in\mathcal{S}(\mathfrak{A})\big\}$.
This set generalizes the classical numerical range in the sense that the numerical range $V(T)$ of a
Hilbert space operator $T$ (considered as an element of a $C^*$-algebra $\mathbb{B}(\mathcal{H})$) coincides with the closure of
its classical numerical range $W(T) = \big\{\langle T\xi, \xi\rangle : \xi\in\mathcal{H}, \|\xi\| = 1\big\}$.
The numerical radius of $x\in\mathfrak{A}$ is defined as
$v(x)=\sup\big\{|\lambda|:\lambda\in V(x)\big\}$.
The study of the numerical range and numerical radius has a long and distinguished history, and has attracted the
attention of several authors; see for instance \cite{bonsall_duncan_1973, gustafson1996numerical, William}.

The numerical range and the numerical radius of operators have been
studied extensively over the last few decades as they are useful in studying and understanding
the role of operators in applications such as numerical analysis, physics and information theory,
differential equations and stability theory of difference approximations for hyperbolic initial-value problems,
see \cite{Axelsson.Lu.Polman, Goldberg.Tadmor, Maroulas.Psarrakos.Tsatsomeros, Psarrakos.Tsatsomeros},
and references therein. Also, the numerical radius is frequently employed as a more reliable
indicator of the rate of convergence of iterative methods than the spectral radius \cite{Axelsson.Lu.Polman, Eiermann}.
It should be mentioned here that the computation of the numerical radius is an optimization problem. For the relevance of the numerical
radius to numerical functional analysis and optimization, we refer the reader to \cite{Ostrovski, Uhlig}, and references therein.

There are many generalizations of the classical numerical range and numerical radius, and there has been a great
deal of interest in their systematic properties and applications;
see \cite{A.K.3, baklouti2018joint, Bo.Co, bourhim2017numerical, GOLDBERG19771, Li.1994, Raji.1, S.B.B.P, S.K.S,
Z.LAA, Zam.2019, Z.MIA, Z.W} and the references therein.

Recently, Bourhim and Mabrouk in \cite{B.M.Positivity}
introduced and studied $a$-numerical range and $a$-numerical radius of elements in $C^*$-algebras.
Also, the authors in \cite{Alahmari.Mabrouk.Zamani} continued the work on the $a$-numerical range and the $a$-numerical radius.
In particular, some ideas from the recent papers are extended.
Let us put $\mathcal{S}_{a}(\mathfrak{A})=\left\{\frac{f}{f(a)}: f\in\mathcal{S}(\mathfrak{A}),\, f(a)\neq0\right\}$.
This set is a non empty, convex and closed subset of the topological dual space of $\mathfrak{A}$,
and is compact if and only if $a$ is invertible in $\mathfrak{A}$;
see \cite[proposition~2.3]{B.M.Positivity}.
For an element $x\in\mathfrak{A}$, let
${\|x\|}_{a}=\sup \left\{\sqrt{f(x^*ax)}: f\in\mathcal{S}_{a}(\mathfrak{A})\right\}$.
Note that $\|\!\cdot\!\|_{\textbf{1}}=\|\!\cdot\!\|$ and $\|x\|_{a}=0$ if and only if $ax=0$.
Notice also that it may happen that $\|x\|_{a}=\infty$ for some $x\in\mathfrak{A}$
due to the lack of compactness of $\mathcal{S}_{a}(\mathfrak{A})$ (see \cite[Example~3.2]{B.M.Positivity}).
From now on we will denote
$\mathfrak{A}^{a}=\left\{x\in\mathfrak{A}: \, \|x\|_{a} <\infty \right\}$.
Observe that $\mathfrak{A}^{a} = \mathfrak{A}$ provided that
$a$ element in the center of $\mathfrak{A}$.
By \cite[Proposition~3.3]{B.M.Positivity} $\|\!\cdot\!\|_{a}$ is a semi-norm on $\mathfrak{A}^{a}$ and moreover
$\|xy\|_{a} \leq \|x\|_{a}\|y\|_{a}$ for all $x, y \in\mathfrak{A}^{a}$. Thus $\mathfrak{A}^{a}$ is a
subalgebra of $\mathfrak{A}$.
For $x\in\mathfrak{A}$, an element $x^{\sharp_{a}}\in\mathfrak{A}$ is called an $a$-adjoint of $x$ if $ax^{\sharp_{a}}=x^*a$.
The set of all elements in $\mathfrak{A}$ that admit $a$-adjoints is denoted by $\mathfrak{A}_{a}$.
Note that $\mathfrak{A}_{a} = \mathfrak{A}$ if $\mathfrak{A}$ is commutative but in general $\mathfrak{A}_{a} \neq \mathfrak{A}$.
Also, neither the existence nor the uniqueness of $a$-adjoint elements is guaranteed.
We also observe that $\mathfrak{A}_{a}$ is a subalgebra of $\mathfrak{A}$ which is neither closed nor dense in $\mathfrak{A}$;
see \cite{B.M.Positivity}. If $x\in\mathfrak{A}_{a}$ and $x^{\sharp_{a}}$ is an $a$-adjoint of it,
then by \cite[Corollary~4.9]{B.M.Positivity}
\begin{align}\label{I.00000}
{\left\|x\right\|}^2_{a} = {\left\|xx^{\sharp_{a}}\right\|}_{a} = {\left\|x^{\sharp_{a}}x\right\|}_{a}= {\left\|x^{\sharp_{a}}\right\|}^2_{a}.
\end{align}
An element $x\in\mathfrak{A}$ is said to be $a$-self-adjoint if $ax$ is self-adjoint, i.e., $ax = x^*a$.
We say that $x$ is $a$-positive if $ax$ is positive.
Every element $x$ in $\mathfrak{A}_{a}$ can be written as $x = y + i z$ where $y$ and $z$ are $a$-self-adjoint but, in general,
this decomposition is not unique. In fact if $x^{\sharp_{a}}$ is an $a$-adjoint of $x$, then
$x = \mathfrak{R}(x) +i\mathfrak{I}(x)$,
where $\mathfrak{R}(x) = \frac{x + x^{\sharp_{a}}}{2}$ and $\mathfrak{I}(x) = \frac{x - x^{\sharp_{a}}}{2i}$
are $a$-real and $a$-imaginary parts of $x$, respectively.
%%%%%%%%%%%%%%%%%%%%%
The $a$-numerical range and $a$-numerical radius of an element $x\in\mathfrak{A}$ are defined by
$V_{a}(x)=\left\{f(ax): f\in\mathcal{S}_{a}(\mathfrak{A})\right\}$
and $v_{a}(x)= \sup\left\{|\lambda|:\lambda\in V_{a}(x)\right\}$, respectively.
Unlike the classical algebraic numerical range, the $a$-numerical range $V_{a}(x)$ of an element $x\in\mathfrak{A}$ may or may not
be closed and/or may or may not be bounded.
%%%%%%%%%%%%%%%%%%%%%%
Originally, these concepts were introduced in \cite{B.M.Positivity} as generalizations
of the $A$-numerical range and $A$-numerical radius for Hilbert space operator $T$ defined by
$W_{A}(T)=\left\{\langle AT\xi, \xi\rangle: \xi\in \mathcal{H}, \, {\|\xi\|}_{A} = 1 \right\}$
and
$w_{A}(T)= \sup\left\{|\lambda|:\lambda\in W_{A}(T)\right\}$, respectively.
Here, $A$ is a positive operator on $\mathcal{H}$ and ${\|\xi\|}_{A} = \sqrt{\langle A\xi, \xi\rangle}$ for all $\xi\in \mathcal{H}$.
In particular, if in the definitions of the $A$-numerical range and
$A$-numerical radius of $T$ we consider $A$ the identity operator on $\mathcal{H}$,
then we get the classical numerical range and numerical radius, respectively, i.e.,
$W_{A}(T) = W(T)$ and $w_{A}(T) = w(T)$.
%%%%%%%%%%%%%%%%
An important and useful identity for the $a$-numerical radius (see \cite[Theorem~4.11]{B.M.Positivity}) is as follows:
\begin{align*}
v_{_{a}}(x)= \displaystyle{\sup_{\theta \in \mathbb{R}}}
\,{\left\|\mathfrak{R}(e^{i\theta}x)\right\|}_{a}.
\end{align*}
By \cite[Proposition~3.3 and Corollary 4.10]{B.M.Positivity}
observe that $v_{a}(\cdot)$ defines a semi-norm on ${\mathfrak{A}}_{a}$, which is equivalent to the $a$-operator
semi-norm ${\|\!\cdot\!\|}_{a}$. Namely, for $x\in{\mathfrak{A}}_{a}$, it holds that
\begin{align}\label{I.000009}
\frac{1}{2}{\|x\|}_{a} \leq v_{_{a}}(x) \leq {\|x\|}_{a}.
\end{align}
The first inequality becomes equality if $ax\neq 0$ and $ax^2 = 0$ (see Remark \ref{R.0.509}) and
the second inequality becomes equality if $x$ is $a$-self-adjoint (see, \cite[Corollary~4.6]{B.M.Positivity}).
%%%%%%%%%%%%%%%%%%%%%

In this paper, we first define a semi-norm on $\mathfrak{A}$,
which generalizes the $a$-operator semi-norm and the $a$-numerical radius
and prove some basic properties. Many inequalities
involving this semi-norm are given.
In addition, we present new upper and lower bounds for the $a$-numerical radii of elements in $\mathfrak{A}$.
These inequalities generalize known numerical radius inequalities.
Finally we define $a$-numerical index of $\mathfrak{A}$ and prove some results concerning it.
%%%%%%%%%%%%%%%%%%%%%%%%%%%%%%%%%%%%%%%%%%%%%%
\section{An extension of the $a$-numerical radius on $C^*$-algebras}
%%%%%%%%%%%%%%%%%%%
In this section, we introduce our new semi-norm on ${\mathfrak{A}}_{a}$,
which generalizes the $a$-operator semi-norm and the $a$-numerical radius, and present basic properties
of this semi-norm and prove inequalities involving it.
%%%%%%%%%%%%%%%%%
First, let us define notions weighted $a$-real and $a$-imaginary parts of elements in ${\mathfrak{A}}_{a}$.
For $t, s \geq 0$ with $t+s>0$, we define the
weighted $a$-real and $a$-imaginary parts of $x\in {\mathfrak{A}}_{a}$ by $\mathfrak{R}_{_{(t, s)}}(x) = tx + sx^{\sharp_{a}}$
and $\mathfrak{I}_{_{(t, s)}}(x) = t(-ix) + s(-ix)^{\sharp_{a}}$, respectively. When $t = s = \frac{1}{2}$, we clearly have
$\mathfrak{R}_{_{(\frac{1}{2}, \frac{1}{2})}}(x) = \mathfrak{R}(x)$ and
$\mathfrak{I}_{_{(\frac{1}{2}, \frac{1}{2})}}(x) = \mathfrak{I}(x)$.
%%%%%%%%%%%%%%%%%%%%%%%%%%%%
\begin{definition}\label{D.1.2}
Let $\mathfrak{A}$ be a unital $C^*$-algebra with the unit $\textbf{1}$ and let $a\in\mathfrak{A}$ be a positive element.
For $t, s \geq 0$ with $t+s>0$, the function $v_{_{(a, (t, s))}}(\cdot)\colon {\mathfrak{A}}_{a} \longrightarrow [0,+\infty)$
is defined as
\begin{align*}
v_{_{(a, (t, s))}}(x)= \displaystyle{\sup_{\theta \in \mathbb{R}}}
\,{\left\|\mathfrak{R}_{_{(t, s)}}(e^{i\theta}x)\right\|}_{a}.
\end{align*}
\end{definition}
%%%%%%%%%%%%%%%%%%%%%%%%
\begin{remark}\label{R.2.2}
For $x\in {\mathfrak{A}}_{a}$,
it is easy to see that $v_{_{(a, (t, s))}}(x)= \displaystyle{\sup_{\theta \in \mathbb{R}}}
\,{\left\|\mathfrak{I}_{_{(t, s)}}(e^{i\theta}x)\right\|}_{a}$.
\end{remark}
%%%%%%%%%%%%%%%%%%%%%%%%%%%%%%%
\begin{remark}\label{R.4.2}
Obviously, $v_{_{(a, (1, 0))}}(x) = v_{_{(a, (0, 1))}}(x) = {\|x\|}_{a}$, and
$v_{_{(a, (\frac{1}{2}, \frac{1}{2}))}}(x) = v_{a}(x)$.
Hence $v_{_{(a, (t, s))}}(\cdot)$ generalizes the $a$-operator semi-norm ${\|\!\cdot\!\|}_{a}$ and the $a$-numerical radius $v_{a}(\cdot)$,
which have been introduced in \cite{B.M.Positivity}.
\end{remark}
%%%%%%%%%%%%%%%%%%
\begin{remark}\label{R.3.2}
Let $\mathfrak{A} = \mathbb{B}(\mathcal{H})$ and let $0\leq \nu \leq 1$.
We have $v_{_{(I, (\nu, 1-\nu))}}(T) = \displaystyle{\sup_{\theta \in \mathbb{R}}}
\left\|\nu e^{i\theta} T + (1-\nu)(e^{i\theta}T)^*\right\|: = w_{_{\nu}}(T)$.
Thus $v_{_{(a, (t, s))}}(\cdot)$ also generalizes the weighted numerical radius $w_{_{\nu}}(\cdot)$,
which has been recently introduced in \cite{S.K.S}.
\end{remark}
%%%%%%%%%%%%%%%%%%%%%%%%%%%%%%%%%
Our first result reads as follows.
%%%%%%%%%%%%%%%%%%%%%%%%
\begin{theorem}\label{t.70.2}
Let $\mathfrak{A}$ be a $C^*$-algebra and let $x\in {\mathfrak{A}}_{a}$. The following statements hold.
\begin{itemize}
\item[(i)] $v_{_{(a, (t, s))}}(x)= \displaystyle{\sup_{\alpha, \beta \in \mathbb{R}, \alpha^2 + \beta^2 = 1}}
\,{\left\|\alpha \,\mathfrak{R}_{_{(t, s)}}(x) + \beta \,\mathfrak{I}_{_{(t, s)}}(x)\right\|}_{a}$.
\item[(ii)] $v_{_{(a, (t, s))}}(x)= \frac{1}{2}\displaystyle{\sup_{\theta,\varphi \in \mathbb{R}}}
\,{\left\|\mathfrak{R}_{_{(t, s)}}\left((e^{i\theta}-ie^{i\varphi})x\right)\right\|}_{a}$.
\end{itemize}
\end{theorem}
\begin{proof}
(i) Let $\theta \in \mathbb{R}$. Put $\alpha = \cos \theta$ and $\beta = -\sin \theta$.
We have
\begin{align*}
\mathfrak{R}_{_{(t, s)}}(e^{i\theta}x) &= te^{i\theta}x + se^{-i\theta}x^{\sharp_{a}}
\\&= \cos \theta\left(tx + sx^{\sharp_{a}}\right) - \sin \theta\left(t(-ix) + s(-ix)^{\sharp_{a}}\right)
\\&= \alpha \,\mathfrak{R}_{_{(t, s)}}(x) + \beta \,\mathfrak{I}_{_{(t, s)}}(x).
\end{align*}
Therefore
\begin{align*}
\displaystyle{\sup_{\theta \in \mathbb{R}}}
\,{\left\|\mathfrak{R}_{_{(t, s)}}(e^{i\theta}x)\right\|}_{a}
= \displaystyle{\sup_{\alpha, \beta \in \mathbb{R}, \alpha^2 + \beta^2 = 1}}
\,{\left\|\alpha \,\mathfrak{R}_{_{(t, s)}}(x) + \beta \,\mathfrak{I}_{_{(t, s)}}(x)\right\|}_{a},
\end{align*}
and hence, by Definition \ref{D.1.2}, we obtain
$v_{_{(a, (t, s))}}(x)= \displaystyle{\sup_{\alpha, \beta \in \mathbb{R}, \alpha^2 + \beta^2 = 1}}
\,{\left\|\alpha \,\mathfrak{R}_{_{(t, s)}}(x) + \beta \,\mathfrak{I}_{_{(t, s)}}(x)\right\|}_{a}$.

(ii) We have
\begin{align*}
v_{_{(a, (t, s))}}(x)&= \frac{1}{2}\displaystyle{\sup_{\theta \in \mathbb{R}}}
\,{\left\|\mathfrak{R}_{_{(t, s)}}(e^{i\theta}x) + \mathfrak{R}_{_{(t, s)}}(e^{i\theta}x)\right\|}_{a}
\\& = \frac{1}{2}\displaystyle{\sup_{\theta \in \mathbb{R}}}
\,{\left\|\mathfrak{R}_{_{(t, s)}}(e^{i\theta}x) + \mathfrak{I}_{_{(t, s)}}(e^{i(\theta+\pi/2)}x)\right\|}_{a}
\\&\leq \frac{1}{2}\displaystyle{\sup_{\theta,\varphi \in \mathbb{R}}}
\,{\left\|\mathfrak{R}_{_{(t, s)}}(e^{i\theta}x) + \mathfrak{I}_{_{(t, s)}}(e^{i\varphi}x)\right\|}_{a}
\\&= \frac{1}{2}\displaystyle{\sup_{\theta,\varphi \in \mathbb{R}}}
\,{\left\|t\left(e^{i\theta} -ie^{i\varphi}\right)x + s\left(\left(e^{i\theta} -ie^{i\varphi}\right)x\right)^{\sharp_{a}}\right\|}_{a}
\\&= \frac{1}{2}\displaystyle{\sup_{\theta,\varphi \in \mathbb{R}}}
\,{\left\|\mathfrak{R}_{_{(t, s)}}\left(\left(e^{i\theta} -ie^{i\varphi}\right)x\right)\right\|}_{a}
\\&\leq \frac{1}{2}\displaystyle{\sup_{\theta,\varphi \in \mathbb{R}}}
\,v_{_{(a, (t, s))}}\left(\left(\left(e^{i\theta} -ie^{i\varphi}\right)x\right)\right)
\\&= \frac{1}{2}\displaystyle{\sup_{\theta,\varphi \in \mathbb{R}}}
\left|e^{i\theta} -ie^{i\varphi}\right|v_{_{(a, (t, s))}}(x)
\\&= \frac{v_{_{(a, (t, s))}}(x)}{2}\displaystyle{\sup_{\theta,\varphi \in \mathbb{R}}}
\sqrt{2-2\sin(\theta-\varphi)}
= v_{_{(a, (t, s))}}(x),
\end{align*}
and so $v_{_{(a, (t, s))}}(x)= \frac{1}{2}\displaystyle{\sup_{\theta,\varphi \in \mathbb{R}}}
\,{\left\|\mathfrak{R}_{_{(t, s)}}\left((e^{i\theta}-ie^{i\varphi})x\right)\right\|}_{a}$.
\end{proof}
%%%%%%%%%%%%%%%%%%%%%%
In the following theorem, we show that $v_{_{(a, (t, s))}}(\cdot)$ is a semi-norm on ${\mathfrak{A}}_{a}$,
which is equivalent to the $a$-operator semi-norm ${\|\!\cdot\!\|}_{a}$.
%%%%%%%%%%%%%%%%%%%%%%%%%%%%%%%%
\begin{theorem}\label{T.6.2}
Let $\mathfrak{A}$ be a $C^*$-algebra.
Then $v_{_{(a, (t, s))}}(\cdot)$ is a semi-norm on ${\mathfrak{A}}_{a}$
and for every $x\in {\mathfrak{A}}_{a}$ the following inequalities hold:
\begin{align*}
\max\{t, s\} {\|x\|}_{a}\leq v_{_{(a, (t, s))}}(x) \leq (t+s){\|x\|}_{a}.
\end{align*}
\end{theorem}
\begin{proof}
The proof that $v_{_{(a, (t, s))}}(\cdot)$ is a semi-norm on ${\mathfrak{A}}_{a}$
is so similar to that of \cite[Theorem~1]{A.K.3} that we omit it.

Now, let $x\in {\mathfrak{A}}_{a}$. By taking $\theta = 0$
in Definition~\ref{D.1.2} and Remark~\ref{R.2.2}, we deduce that
\begin{align}\label{I.1.T.6.28}
v_{_{(a, (t, s))}}(x) \geq {\left\|tx + sx^{\sharp_{a}}\right\|}_{a}
\quad \mbox{and} \quad  v_{_{(a, (t, s))}}(x)\geq {\left\|-itx + isx^{\sharp_{a}}\right\|}_{a}.
\end{align}
Thus, by \eqref{I.1.T.6.28} and the triangle inequality for the semi-norm ${\|\!\cdot\!\|}_{a}$, we have
\begin{align*}
v_{_{(a, (t, s))}}(x) &\geq \frac{{\left\|tx + sx^{\sharp_{a}}\right\|}_{a} + {\left\|-itx + isx^{\sharp_{a}}\right\|}_{a}}{2}
\\& \geq \frac{{\left\|(tx + sx^{\sharp_{a}}) + i(-itx + isx^{\sharp_{a}})\right\|}_{a}}{2}
= t{\|x\|}_{a},
\end{align*}
and hence
\begin{align}\label{I.2.T.6.29}
t{\|x\|}_{a} \leq v_{_{(a, (t, s))}}(x).
\end{align}
By a similar argument, we obtain
\begin{align}\label{I.5.T.6.2}
s{\|x\|}_{a} \leq v_{_{(a, (t, s))}}(x).
\end{align}
Now, \eqref{I.2.T.6.29} and \eqref{I.5.T.6.2} yield that
\begin{align}\label{I.6.T.6.2}
\max\{t, s\} {\|x\|}_{a}\leq v_{_{(a, (t, s))}}(x).
\end{align}
Furthermore, by the triangle inequality for the semi-norm ${\|\!\cdot\!\|}_{a}$ and \eqref{I.00000}, we have
\begin{align*}
v_{_{(a, (t, s))}}(x)&= \displaystyle{\sup_{\theta \in \mathbb{R}}}
\,{\left\|te^{i\theta}x + se^{-i\theta}x^{\sharp_{a}}\right\|}_{a}.
\\& \leq \displaystyle{\sup_{\theta \in \mathbb{R}}}
\left(t{\left\|x\right\|}_{a} + s{\left\|x^{\sharp_{a}}\right\|}_{a}\right)
= t{\left\|x\right\|}_{a} + s{\left\|x\right\|}_{a},
\end{align*}
and so
\begin{align}\label{I.7.T.6.2}
v_{_{(a, (t, s))}}(x)\leq (t+s){\left\|x\right\|}_{a}.
\end{align}
From \eqref{I.6.T.6.2} and \eqref{I.7.T.6.2}, we deduce the desired result.
\end{proof}
%%%%%%%%%%%%%%%%%%%%%%%%%%%
%%%%%%%%%%%%%%%%%%%%%%%%%%%%
\begin{remark}\label{P.7.2}
For $x\in {\mathfrak{A}}_{a}$, by \eqref{I.00000}, we have
\begin{align*}
v_{_{(a, (t, s))}}(x^{\sharp_{a}})&= \displaystyle{\sup_{\theta \in \mathbb{R}}}
\,{\left\|te^{i\theta}x^{\sharp_{a}} + se^{-i\theta}(x^{\sharp_{a}})^{\sharp_{a}}\right\|}_{a}
\\&= \displaystyle{\sup_{\theta \in \mathbb{R}}}
\,{\left\|\left(te^{-i\theta}x + se^{i\theta}x^{\sharp_{a}}\right)^{\sharp_{a}}\right\|}_{a}
\\& = \displaystyle{\sup_{\theta \in \mathbb{R}}}
\,{\left\|te^{-i\theta}x + se^{i\theta}x^{\sharp_{a}}\right\|}_{a}
= v_{_{(a, (t, s))}}(x),
\end{align*}
and hence $v_{_{(a, (t, s))}}(x^{\sharp_{a}}) = v_{_{(a, (t, s))}}(x)$.
\end{remark}
%%%%%%%%%%%%%%%%%%%%%%%%%%
%%%%%%%%%%%%%%%%%%%%%%%%
In the following result, we give a condition equivalent to $v_{_{(a, (t, s))}}(x)=\max\{t, s\} {\|x\|}_{a}$.
%%%%%%%%%%%%%%%%%%%%%%%
\begin{theorem}\label{t.71.2}
Let $\mathfrak{A}$ be a $C^*$-algebra and let $x\in {\mathfrak{A}}_{a}$. The following conditions are equivalent:
\begin{itemize}
\item[(i)] ${\left\|\mathfrak{R}_{_{(t, s)}}(e^{i\theta}x)\right\|}_{a} = \max\{t, s\} {\|x\|}_{a}$ for all $\theta \in \mathbb{R}$.
\item[(ii)] $v_{_{(a, (t, s))}}(x)=\max\{t, s\} {\|x\|}_{a}$.
\end{itemize}
\end{theorem}
\begin{proof}
(i)$\Rightarrow$(ii) The implication is trivial.

(ii)$\Rightarrow$(i) Suppose (ii) holds.
Let $\theta\in \mathbb{R}$.
By Definition \ref{D.1.2}, Remark \ref{R.2.2} and \eqref{I.6.T.6.2} we have
\begin{align*}
\max\{t, s\} {\|x\|}_{a} & = v_{_{(a, (t, s))}}(x)
\\& \geq \frac{1}{2}\max\left\{\max\{t, s\} {\|x\|}_{a}, {\left\|\mathfrak{R}_{_{(t, s)}}(e^{i\theta}x)\right\|}_{a}\right\}
\\& \qquad \qquad + \frac{1}{2}\max\left\{\max\{t, s\} {\|x\|}_{a}, {\left\|\mathfrak{I}_{_{(t, s)}}(e^{i\theta}x)\right\|}_{a}\right\}
\\& \geq \frac{1}{4}\left(2\max\{t, s\} {\|x\|}_{a} + {\left\|\mathfrak{R}_{_{(t, s)}}(e^{i\theta}x)\right\|}_{a}
+ {\left\|\mathfrak{I}_{_{(t, s)}}(e^{i\theta}x)\right\|}_{a}\right)
\\& \qquad \qquad + \frac{1}{4}\left|\max\{t, s\} {\|x\|}_{a}- {\left\|\mathfrak{R}_{_{(t, s)}}(e^{i\theta}x)\right\|}_{a}\right|
\\& \geq \frac{1}{2}\max\{t, s\} {\|x\|}_{a}
+ \frac{1}{4} {\left\|\mathfrak{R}_{_{(t, s)}}(e^{i\theta}x) + i\mathfrak{I}_{_{(t, s)}}(e^{i\theta}x)\right\|}_{a}
\\& \qquad \qquad + \frac{1}{4}\left|\max\{t, s\} {\|x\|}_{a} - {\left\|\mathfrak{R}_{_{(t, s)}}(e^{i\theta}x)\right\|}_{a}\right|
\\& = \frac{1}{2}\max\{t, s\} {\|x\|}_{a} + \frac{t}{2}{\|x\|}_{a}
+ \frac{1}{4}\left|\max\{t, s\} {\|x\|}_{a} - {\left\|\mathfrak{R}_{_{(t, s)}}(e^{i\theta}x)\right\|}_{a}\right|,
\end{align*}
and so
\begin{align}\label{I.1.t.71.2}
t{\|x\|}_{a} + \frac{1}{2}\left|\max\{t, s\} {\|x\|}_{a} - {\left\|\mathfrak{R}_{_{(t, s)}}(e^{i\theta}x)\right\|}_{a}\right|
\leq \max\{t, s\} {\|x\|}_{a}.
\end{align}
Further, by a similar argument, we have
\begin{align}\label{I.2.t.71.2}
s{\|x\|}_{a} + \frac{1}{2}\left|\max\{t, s\} {\|x\|}_{a} - {\left\|\mathfrak{R}_{_{(t, s)}}(e^{i\theta}x)\right\|}_{a}\right|
\leq \max\{t, s\} {\|x\|}_{a}.
\end{align}
Utilizing \eqref{I.1.t.71.2} and \eqref{I.2.t.71.2} we obtain
\begin{align*}
\max\{t, s\} {\|x\|}_{a} + \frac{1}{2}\left|\max\{t, s\} {\|x\|}_{a} - {\left\|\mathfrak{R}_{_{(t, s)}}(e^{i\theta}x)\right\|}_{a}\right|
\leq \max\{t, s\} {\|x\|}_{a},
\end{align*}
and hence ${\left\|\mathfrak{R}_{_{(t, s)}}(e^{i\theta}x)\right\|}_{a} = \max\{t, s\} {\|x\|}_{a}$.
\end{proof}
%%%%%%%%%%%%%%%%%%%%%%%%%%%%%
In the following theorem, a refinement of the inequality \eqref{I.7.T.6.2} is given.
%%%%%%%%%%%%%%%%%%%%%%%%%%%%%%
\begin{theorem}\label{T.2.3}
Let $\mathfrak{A}$ be a $C^*$-algebra and let $x\in {\mathfrak{A}}_{a}$. Then
\begin{align*}
v_{_{(a, (t, s))}}(x)\leq \sqrt{(t^2+ s^2){\|x\|}^2_{a} + 2ts\,v_{_{a}}\left(x^2\right)}
\leq (t+s){\|x\|}_{a}.
\end{align*}
\end{theorem}
\begin{proof}
For every $\theta\in \mathbb{R}$, since $ax^{\sharp_{a}} = x^*a$, it is easy to see that
\begin{align*}
a\big(e^{2i\theta}x^2 + e^{-2i\theta}(x^2)^{\sharp_{a}}\big)
= \big(e^{2i\theta}x\big(x^{\sharp_{a}}\big)^{\sharp_{a}} + e^{-2i\theta}(x^{\sharp_{a}})^2\big)^*a.
\end{align*}
Thus $e^{2i\theta}x^2 + e^{-2i\theta}(x^2)^{\sharp_{a}}$ is an $a$-adjoint of $e^{2i\theta}x\big(x^{\sharp_{a}}\big)^{\sharp_{a}} + e^{-2i\theta}(x^{\sharp_{a}})^2$.
So, by \eqref{I.00000} and the triangle inequality for the semi-norm ${\|\!\cdot\!\|}_{a}$, we have
\begin{align*}
v_{_{(a, (t, s))}}^2(x)&= \displaystyle{\sup_{\theta \in \mathbb{R}}}
\,{\left\|te^{i\theta}x + se^{-i\theta}x^{\sharp_{a}}\right\|}^2_{a}
\\& = \displaystyle{\sup_{\theta \in \mathbb{R}}}
\,{\left\|\big(te^{i\theta}x + se^{-i\theta}x^{\sharp_{a}}\big)\big(te^{i\theta}x + se^{-i\theta}x^{\sharp_{a}}\big)^{\sharp_{a}}\right\|}_{a}
\\& = \displaystyle{\sup_{\theta \in \mathbb{R}}}
\,{\left\|t^2xx^{\sharp_{a}} + s^2x^{\sharp_{a}}\big(x^{\sharp_{a}}\big)^{\sharp_{a}} + ts\big(e^{2i\theta}x\big(x^{\sharp_{a}}\big)^{\sharp_{a}} + e^{-2i\theta}(x^{\sharp_{a}})^2\big)\right\|}_{a}
\\& \leq \displaystyle{\sup_{\theta \in \mathbb{R}}}
\,\left({\left\|t^2xx^{\sharp_{a}} + s^2x^{\sharp_{a}}\big(x^{\sharp_{a}}\big)^{\sharp_{a}}\right\|}_{a}
+ ts{\left\|e^{2i\theta}x\big(x^{\sharp_{a}}\big)^{\sharp_{a}} + e^{-2i\theta}(x^{\sharp_{a}})^2\right\|}_{a}\right)
\\& \leq {\left\|t^2xx^{\sharp_{a}} + s^2x^{\sharp_{a}}\big(x^{\sharp_{a}}\big)^{\sharp_{a}}\right\|}_{a}
+ ts\,\displaystyle{\sup_{\theta \in \mathbb{R}}}
\,{\left\|e^{2i\theta}x\big(x^{\sharp_{a}}\big)^{\sharp_{a}} + e^{-2i\theta}(x^{\sharp_{a}})^2\right\|}_{a}
\\& = {\left\|t^2xx^{\sharp_{a}} + s^2x^{\sharp_{a}}\big(x^{\sharp_{a}}\big)^{\sharp_{a}}\right\|}_{a}
+ ts\,\displaystyle{\sup_{\theta \in \mathbb{R}}}
\,{\left\|\left(e^{2i\theta}x\big(x^{\sharp_{a}}\big)^{\sharp_{a}} + e^{-2i\theta}(x^{\sharp_{a}})^2\right)^{\sharp_{a}}\right\|}_{a}
\\& = {\left\|t^2xx^{\sharp_{a}} + s^2x^{\sharp_{a}}\big(x^{\sharp_{a}}\big)^{\sharp_{a}}\right\|}_{a}
+ ts\,\displaystyle{\sup_{\theta \in \mathbb{R}}}
\,{\left\|e^{2i\theta}x^2 + e^{-2i\theta}(x^2)^{\sharp_{a}}\right\|}_{a}
\\& = {\left\|t^2xx^{\sharp_{a}} + s^2x^{\sharp_{a}}\big(x^{\sharp_{a}}\big)^{\sharp_{a}}\right\|}_{a}
+ 2ts\,v_{_{a}}\left(x^2\right)
\\& \leq t^2{\left\|xx^{\sharp_{a}}\right\|}_{a} + s^2{\left\|x^{\sharp_{a}}\big(x^{\sharp_{a}}\big)^{\sharp_{a}}\right\|}_{a}
+ 2ts\,v_{_{a}}\left(x^2\right)
\\& = (t^2+ s^2){\|x\|}^2_{a} + 2ts\,v_{_{a}}\left(x^2\right),
\end{align*}
and hence
\begin{align}\label{I.1.T.2.3}
v_{_{(a, (t, s))}}^2(x)\leq (t^2+ s^2){\|x\|}^2_{a} + 2ts\,v_{_{a}}\left(x^2\right).
\end{align}
Furthermore, since $v_{_{a}}\left(x^2\right)\leq {\|x^2\|}_{a}\leq {\|x\|}^2_{a}$, we have
\begin{align}\label{I.2.T.2.3}
(t^2+ s^2){\|x\|}^2_{a} + 2ts\,v_{_{a}}\left(x^2\right) \leq (t+s)^2{\|x\|}^2_{a}.
\end{align}
Now, from \eqref{I.1.T.2.3} and \eqref{I.2.T.2.3}, we deduce the desired result.
\end{proof}
%%%%%%%%%%%%%%%%%%%%%%%%%%%%%%%%%%%%%%%%%%%%
\begin{corollary}\label{C.2.4}
Let $\mathfrak{A}$ be a $C^*$-algebra.
If $x\in {\mathfrak{A}}_{a}$ is such that $v_{_{(a, (t, s))}}(x) = (t+s){\|x\|}_{a}$, then
${\|x^2\|}_{a} = {\|x\|}^2_{a}$.
\end{corollary}
\begin{proof}
Let $x\in {\mathfrak{A}}_{a}$ such that $v_{_{(a, (t, s))}}(x) = (t+s){\|x\|}_{a}$.
By Theorem \ref{T.2.3} we have
\begin{align*}
(t+s)^2{\|x\|}^2_{a} &\leq (t^2+ s^2){\|x\|}^2_{a} + 2ts\,v_{_{a}}\left(x^2\right)
\\& \leq (t^2+ s^2){\|x\|}^2_{a} + 2ts{\|x^2\|}_{a}
\\& \leq (t^2+ s^2){\|x\|}^2_{a} + 2ts{\|x\|}^2_{a} = (t+s)^2{\|x\|}^2_{a},
\end{align*}
which implies ${\|x^2\|}_{a} = {\|x\|}^2_{a}$.
\end{proof}
%%%%%%%%%%%%%%%%%%%%%%%
Our next result extends and refines an inequality for the numerical radius of Hilbert space
operators obtained by Kittaneh in \cite{K.2005}.
%%%%%%%%%%%%%%%%%%%%%%%%%%%%%%
\begin{theorem}\label{T.5.3}
Let $\mathfrak{A}$ be a $C^*$-algebra and let $x\in {\mathfrak{A}}_{a}$. Then
\begin{align*}
ts {\left\|xx^{\sharp_{a}} + x^{\sharp_{a}}x\right\|}_{a}
+ \frac{1}{2}\displaystyle{\sup_{\theta \in \mathbb{R}}}\left|{\left\|\mathfrak{R}_{_{(t, s)}}(e^{i\theta}x)\right\|}^2_{a} -
{\left\|\mathfrak{I}_{_{(t, s)}}(e^{i\theta}x)\right\|}^2_{a}\right|
\leq v_{_{(a, (t, s))}}^2(x).
\end{align*}
\end{theorem}
\begin{proof}
Let $\theta\in \mathbb{R}$.
By Definition \ref{D.1.2} and Remark \ref{R.2.2} we have
\begin{align}\label{F.1.T.5.3}
v_{_{(a, (t, s))}}(x)\geq \max\left\{{\left\|\mathfrak{R}_{_{(t, s)}}(e^{i\theta}x)\right\|}_{a}, {\left\|\mathfrak{I}_{_{(t, s)}}(e^{i\theta}x)\right\|}_{a}\right\}.
\end{align}
Further, it is easy to see that
\begin{align}\label{F.2.T.5.3}
\mathfrak{R}_{_{(t, s)}}^2(e^{i\theta}x) + \mathfrak{I}_{_{(t, s)}}^2(e^{i\theta}x) = 2ts\big(xx^{\sharp_{a}} + x^{\sharp_{a}}x\big).
\end{align}
By \eqref{F.1.T.5.3} and \eqref{F.2.T.5.3} we have
\begin{align*}
v_{_{(a, (t, s))}}^2(x) &\geq \max\left\{{\left\|\mathfrak{R}_{_{(t, s)}}(e^{i\theta}x)\right\|}^2_{a},
{\left\|\mathfrak{I}_{_{(t, s)}}(e^{i\theta}x)\right\|}^2_{a}\right\}
\\& = \frac{{\left\|\mathfrak{R}_{_{(t, s)}}(e^{i\theta}x)\right\|}^2_{a} +
{\left\|\mathfrak{I}_{_{(t, s)}}(e^{i\theta}x)\right\|}^2_{a}}{2}
+ \frac{\left|{\left\|\mathfrak{R}_{_{(t, s)}}(e^{i\theta}x)\right\|}^2_{a} -
{\left\|\mathfrak{I}_{_{(t, s)}}(e^{i\theta}x)\right\|}^2_{a}\right|}{2}
\\& \geq \frac{{\left\|\mathfrak{R}_{_{(t, s)}}^2(e^{i\theta}x)\right\|}_{a} +
{\left\|\mathfrak{I}_{_{(t, s)}}^2(e^{i\theta}x)\right\|}_{a}}{2}
+ \frac{\left|{\left\|\mathfrak{R}_{_{(t, s)}}(e^{i\theta}x)\right\|}^2_{a} -
{\left\|\mathfrak{I}_{_{(t, s)}}(e^{i\theta}x)\right\|}^2_{a}\right|}{2}
\\& \geq \frac{{\left\|\mathfrak{R}_{_{(t, s)}}^2(e^{i\theta}x) + \mathfrak{I}_{_{(t, s)}}^2(e^{i\theta}x)\right\|}_{a}}{2}
+ \frac{\left|{\left\|\mathfrak{R}_{_{(t, s)}}(e^{i\theta}x)\right\|}^2_{a} -
{\left\|\mathfrak{I}_{_{(t, s)}}(e^{i\theta}x)\right\|}^2_{a}\right|}{2}
\\& = ts {\left\|xx^{\sharp_{a}} + x^{\sharp_{a}}x\right\|}_{a}
+ \frac{\left|{\left\|\mathfrak{R}_{_{(t, s)}}(e^{i\theta}x)\right\|}^2_{a} -
{\left\|\mathfrak{I}_{_{(t, s)}}(e^{i\theta}x)\right\|}^2_{a}\right|}{2}.
\end{align*}
Thus
\begin{align*}
v_{_{(a, (t, s))}}^2(x) \geq ts {\left\|xx^{\sharp_{a}} + x^{\sharp_{a}}x\right\|}_{a}
+ \frac{\left|{\left\|\mathfrak{R}_{_{(t, s)}}(e^{i\theta}x)\right\|}^2_{a} -
{\left\|\mathfrak{I}_{_{(t, s)}}(e^{i\theta}x)\right\|}^2_{a}\right|}{2}.
\end{align*}
Taking the supremum over $\theta\in \mathbb{R}$ in the above inequality, we deduce the desired result.
\end{proof}
%%%%%%%%%%%%%%%%%%%%%%%%
%%%%%%%%%%%%%%%%%%%%%%%%
\section{Upper and lower bounds for the $a$-numerical radius}
%%%%%%%%%%%%%%%%%%%%%%
In this section, we derive new upper and lower bounds for the $a$-numerical radii of elements in $C^*$-algebras.
We first establish a considerable improvement of inequality $v_{a}(x) \leq {\|x\|}_{a}$.
%%%%%%%%%%%%%%%%%%%%%%%%%%%%%
\begin{theorem}\label{T.0.3}
Let $\mathfrak{A}$ be a $C^*$-algebra and let $x\in{\mathfrak{A}}_{a}$. Then
\begingroup\makeatletter\def\f@size{10}\check@mathfonts
\begin{align*}
v^4_{a}(x) \leq \frac{1}{4}v^2_{a}(x^2)
+ \frac{1}{8}v_{a}\Big(x^2(xx^{\sharp_{a}} + x^{\sharp_{a}}x) + (xx^{\sharp_{a}} + x^{\sharp_{a}}x)x^2\Big)
+ \frac{1}{16}{\Big\|(xx^{\sharp_{a}} + x^{\sharp_{a}}x)^2\Big\|}_{a}.
\end{align*}
\endgroup
\end{theorem}
\begin{proof}
Let $\theta\in \mathbb{R}$. Easy computations show that $\frac{e^{i\theta}x + e^{-i\theta}x^{\sharp_{a}}}{2}$
and $\left(\frac{e^{i\theta}x + e^{-i\theta}x^{\sharp_{a}}}{2}\right)^2$
are $a$-self-adjoint and so, by \cite[Corollary 4.9]{B.M.Positivity}, we get
\begingroup\makeatletter\def\f@size{10}\check@mathfonts
\begin{align}\label{F.1.T.0.3}
{\left\|\frac{e^{i\theta}x + e^{-i\theta}x^{\sharp_{a}}}{2}\right\|}^4_{a} =
{\left\|\left(\frac{e^{i\theta}x + e^{-i\theta}x^{\sharp_{a}}}{2}\right)^4\right\|}_{a}.
\end{align}
\endgroup
Put $M := xx^{\sharp_{a}} + x^{\sharp_{a}}x$. Since $ax^{\sharp_{a}} = x^*a$, it is easy to see that $a(x^{\sharp_{a}})^2 = (x^2)^*a$ and
$a\big(M(x^{\sharp_{a}})^2 + (x^{\sharp_{a}})^2M\big) = \big(Mx^2 + x^2M\big)^*a$.
Thus $(x^{\sharp_{a}})^2$ is an $a$-adjoint of $x^2$ and $(x^{\sharp_{a}})^2M + M(x^{\sharp_{a}})^2$ is an $a$-adjoint of $x^2M + Mx^2$.
We have
\begingroup\makeatletter\def\f@size{10}\check@mathfonts
\begin{align*}
&16\left(\frac{e^{i\theta}x + e^{-i\theta}x^{\sharp_{a}}}{2}\right)^4
= \Big(e^{2i\theta}x^2 + e^{-2i\theta}(x^{\sharp_{a}})^2 + M\Big)^2
\\& = 4\left(\frac{e^{2i\theta}x^2 + e^{-2i\theta}(x^{\sharp_{a}})^2}{2}\right)^2
+ 2\left(\frac{e^{2i\theta}\big(x^2M + Mx^2\big) + e^{-2i\theta}\big((x^{\sharp_{a}})^2M + M(x^{\sharp_{a}})^2\big)}{2}\right) + M^2,
\end{align*}
\endgroup
and therefore by \eqref{F.1.T.0.3},
\begingroup\makeatletter\def\f@size{10}\check@mathfonts
\begin{align*}
&16{\left\|\frac{e^{i\theta}x + e^{-i\theta}x^{\sharp_{a}}}{2}\right\|}^4_{a} = 16{\left\|\left(\frac{e^{i\theta}x + e^{-i\theta}x^{\sharp_{a}}}{2}\right)^4\right\|}_{a}
\\& = {\left\|4\left(\frac{e^{2i\theta}x^2 + e^{-2i\theta}(x^{\sharp_{a}})^2}{2}\right)^2
+ 2\left(\frac{e^{2i\theta}\big(x^2M + Mx^2\big) + e^{-2i\theta}\big((x^{\sharp_{a}})^2M + M(x^{\sharp_{a}})^2\big)}{2}\right) + M^2\right\|}_{a}
\\& \leq 4{\left\|\left(\frac{e^{2i\theta}x^2 + e^{-2i\theta}(x^{\sharp_{a}})^2}{2}\right)^2\right\|}_{a}
+ 2{\left\|\frac{e^{2i\theta}\big(x^2M + Mx^2\big) + e^{-2i\theta}\big((x^{\sharp_{a}})^2M + M(x^{\sharp_{a}})^2\big)}{2}\right\|}_{a}
+ {\left\|M^2\right\|}_{a}
\\& = 4{\left\|\frac{e^{2i\theta}x^2 + e^{-2i\theta}(x^{\sharp_{a}})^2}{2}\right\|}^2_{a}
+ 2{\left\|\frac{e^{2i\theta}\big(x^2M + Mx^2\big) + e^{-2i\theta}\big((x^{\sharp_{a}})^2M + M(x^{\sharp_{a}})^2\big)}{2}\right\|}_{a}
+ {\left\|M^2\right\|}_{a}
\\&\qquad \qquad \qquad \qquad \qquad \qquad \qquad \qquad \qquad \qquad
\Big(\mbox{since $\frac{e^{2i\theta}x^2 + e^{-2i\theta}(x^{\sharp_{a}})^2}{2}$ is $a$-self-adjoint}\Big)
\\& \leq 4v^2_{a}(x^2) + 2 v_{a}\Big(x^2M + Mx^2\Big) + {\left\|M^2\right\|}_{a}.
\end{align*}
\endgroup
Hence
\begingroup\makeatletter\def\f@size{10}\check@mathfonts
\begin{align}\label{F.2.T.0.3}
16{\left\|\frac{e^{i\theta}x + e^{-i\theta}x^{\sharp_{a}}}{2}\right\|}^4_{a} \leq 4v^2_{a}(x^2) + 2 v_{a}\Big(x^2M + Mx^2\Big) + {\left\|M^2\right\|}_{a}.
\end{align}
\endgroup
Taking the supremum over $\theta\in \mathbb{R}$ in \eqref{F.2.T.0.3}, we deduce that
\begingroup\makeatletter\def\f@size{10}\check@mathfonts
\begin{align*}
16v^4_{a}(x) \leq 4v^2_{a}(x^2) + 2 v_{a}\big(x^2M + Mx^2\big) + {\left\|M^2\right\|}_{a}.
\end{align*}
\endgroup
\end{proof}
%%%%%%%%%%%%%%%%%%%%%%%%%%%%%%
\begin{remark}
From Theorem \ref{T.0.3}, \eqref{I.000009} and \eqref{I.00000} we have
\begingroup\makeatletter\def\f@size{10}\check@mathfonts
\begin{align*}
v^4_{a}(x) &\leq \frac{1}{4}{\|x^2\|}^2_{a}
+ \frac{1}{8}{\Big\|x^2(xx^{\sharp_{a}} + x^{\sharp_{a}}x) + (xx^{\sharp_{a}} + x^{\sharp_{a}}x)x^2\Big\|}_{a}
+ \frac{1}{16}{\Big\|xx^{\sharp_{a}} + x^{\sharp_{a}}x\Big\|}^2_{a}
\\& \leq \frac{1}{4}{\|x\|}^4_{a}
+ \frac{1}{8}\Big({\|x^2\|}_{a}{\|xx^{\sharp_{a}} + x^{\sharp_{a}}x\|}_{a} + {\|xx^{\sharp_{a}} + x^{\sharp_{a}}x\|}_{a}{\|x^2\|}_{a}\Big)
+ \frac{1}{16}({\|xx^{\sharp_{a}}\|}_{a} + {\|x^{\sharp_{a}}x\|}_{a})^2
\\& \leq \frac{1}{4}{\|x\|}^4_{a}
+ \frac{1}{8}\Big({\|x\|}^2_{a}({\|xx^{\sharp_{a}}\|}_{a} + {\|x^{\sharp_{a}}x\|}_{a}) + ({\|xx^{\sharp_{a}}\|}_{a} + {\|x^{\sharp_{a}}x\|}_{a}){\|x\|}^2_{a}\Big)
+ \frac{1}{16}(2{\|x\|}^2_{a})^2
\\& = \frac{1}{4}{\|x\|}^4_{a}
+ \frac{1}{8}\Big(2{\|x\|}^4_{a} + 2{\|x\|}^4_{a}\Big)
+ \frac{1}{4}{\|x\|}^4_{a} = {\|x\|}^4_{a}.
\end{align*}
\endgroup
Therefore, Theorem \ref{T.0.3} is a more precise estimate of the $a$-numerical radius $v_{a}(x) \leq {\|x\|}_{a}$.
\end{remark}
%%%%%%%%%%%%%%%%%%%
%%%%%%%%%%%%%%%%%%
\begin{theorem}\label{C.00.3}
Let $\mathfrak{A}$ be a $C^*$-algebra and let $x\in{\mathfrak{A}}_{a}$. Then
\begin{align*}
v^2_{a}(x) \leq \frac{1}{2}v_{a}(x^2) + \frac{1}{4}{\Big\|xx^{\sharp_{a}} + x^{\sharp_{a}}x\Big\|}_{a}.
\end{align*}
\end{theorem}
\begin{proof}
Let $\theta\in \mathbb{R}$ and $M = xx^{\sharp_{a}} + x^{\sharp_{a}}x$.
We have
\begingroup\makeatletter\def\f@size{9}\check@mathfonts
\begin{align*}
\frac{e^{i\theta}\big(x^2M + Mx^2\big)
+ e^{-i\theta}\big((x^{\sharp_{a}})^2M + M(x^{\sharp_{a}})^2\big)}{2}
= \frac{e^{i\theta}x^2 + e^{-i\theta}(x^{\sharp_{a}})^2}{2}M
+ M\frac{e^{i\theta}x^2 + e^{-i\theta}(x^{\sharp_{a}})^2}{2}.
\end{align*}
\endgroup
Since $(x^{\sharp_{a}})^2$ is an $a$-adjoint of $x^2$ we get
\begingroup\makeatletter\def\f@size{10}\check@mathfonts
\begin{align*}
&{\left\|\frac{e^{i\theta}\big(x^2M + Mx^2\big)
+ e^{-i\theta}\big((x^{\sharp_{a}})^2M + M(x^{\sharp_{a}})^2\big)}{2}\right\|}_{a}
\\& = {\left\|\frac{e^{i\theta}x^2 + e^{-i\theta}(x^{\sharp_{a}})^2}{2}M
+ M\frac{e^{i\theta}x^2 + e^{-i\theta}(x^{\sharp_{a}})^2}{2}\right\|}_{a}
\\&\leq {\left\|\frac{e^{i\theta}x^2 + e^{-i\theta}(x^{\sharp_{a}})^2}{2}\right\|}_{a}
{\|M\|}_{a}
+ {\|M\|}_{a}{\left\|\frac{e^{i\theta}x^2 + e^{-i\theta}(x^{\sharp_{a}})^2}{2}\right\|}_{a}
\\& \leq v_{a}(x^2){\|M\|}_{a} + {\|M\|}_{a}v_{a}(x^2) = 2v_{a}(x^2){\|M\|}_{a},
\end{align*}
\endgroup
and hence
\begingroup\makeatletter\def\f@size{10}\check@mathfonts
\begin{align}\label{F.1.C.00.3}
{\left\|\frac{e^{i\theta}\big(x^2M + Mx^2\big)
+ e^{-i\theta}\big((x^{\sharp_{a}})^2M + M(x^{\sharp_{a}})^2\big)}{2}\right\|}_{a}
\leq 2v_{a}(x^2){\|M\|}_{a}.
\end{align}
\endgroup
Since $(x^{\sharp_{a}})^2M + M(x^{\sharp_{a}})^2$ is an $a$-adjoint of $x^2M + Mx^2$,
by taking the supremum over $\theta\in \mathbb{R}$ in \eqref{F.1.C.00.3}, we arrive at
\begingroup\makeatletter\def\f@size{10}\check@mathfonts
\begin{align}\label{F.2.C.00.3}
v_{a}\big(x^2M + Mx^2\big) \leq 2v_{a}(x^2){\|M\|}_{a}.
\end{align}
\endgroup
Now, by Theorem \ref{T.0.3} and \eqref{F.2.C.00.3}, we have
\begingroup\makeatletter\def\f@size{10}\check@mathfonts
\begin{align*}
v^4_{a}(x) &\leq \frac{1}{4}v^2_{a}(x^2) + \frac{1}{8} v_{a}\big(x^2M + Mx^2\big) + \frac{1}{16}{\|M^2\|}_{a}
\\& \leq \frac{1}{4}v^2_{a}(x^2) + \frac{1}{4}v_{a}(x^2){\|M\|}_{a} + \frac{1}{16}{\|M\|}^2_{a}
\\& = \left(\frac{1}{2}v_{a}(x^2) + \frac{1}{4}{\|M\|}_{a}\right)^2
\end{align*}
\endgroup
and so we deduce the desired result.
\end{proof}
%%%%%%%%%%%%%%%%%%%%%%
\begin{remark}\label{C.0.49}
Theorem \ref{C.00.3} is an extension of \cite[Theorem 2.4]{A.K.2}.
\end{remark}
%%%%%%%%%%%%%%%%%%%
%%%%%%%%%%%%%%%%%%%%%%
\begin{corollary}\label{C.0.5}
Let $\mathfrak{A}$ be a $C^*$-algebra and let $x\in{\mathfrak{A}}_{a}$. Then
\begin{align*}
\frac{1}{4}{\big\|xx^{\sharp_{a}} + x^{\sharp_{a}}x\big\|}_{a}
\leq v^2_{a}(x)\leq \frac{1}{2}{\big\|xx^{\sharp_{a}} + x^{\sharp_{a}}x\big\|}_{a}.
\end{align*}
\end{corollary}
\begin{proof}
By Theorem \ref{C.00.3} and the power inequality for the $a$-numerical radius, we have
\begin{align*}
v^2_{a}(x) \leq \frac{1}{2}v_{a}(x^2) + \frac{1}{4}{\Big\|xx^{\sharp_{a}} + x^{\sharp_{a}}x\Big\|}_{a}
\leq \frac{1}{2}v^2_{a}(x) + \frac{1}{4}{\Big\|xx^{\sharp_{a}} + x^{\sharp_{a}}x\Big\|}_{a}.
\end{align*}
Thus
\begin{align}\label{I.1.C.0.5}
v^2_{a}(x) \leq \frac{1}{2}{\Big\|xx^{\sharp_{a}} + x^{\sharp_{a}}x\Big\|}_{a}.
\end{align}
Also, by Theorem \ref{T.5.3} with $t=s=\frac{1}{2}$, we have
\begin{align}\label{I.2.C.0.5}
\frac{1}{4}{\big\|xx^{\sharp_{a}} + x^{\sharp_{a}}x\big\|}_{a}
+ \frac{1}{8}\displaystyle{\sup_{\theta \in \mathbb{R}}}
\Big|{\big\|e^{i\theta}x + e^{-i\theta}x^{\sharp_{a}}\big\|}^2_{a} - {\big\|e^{i\theta}x - e^{-i\theta}x^{\sharp_{a}}\big\|}^2_{a}\Big|\leq v^2_{a}(x),
\end{align}
and hence
\begin{align}\label{I.3.C.0.5}
\frac{1}{4}{\big\|xx^{\sharp_{a}} + x^{\sharp_{a}}x\big\|}_{a} \leq v^2_{a}(x).
\end{align}
Now, from \eqref{I.1.C.0.5} and \eqref{I.3.C.0.5}, we deduce the desired result.
\end{proof}
%%%%%%%%%%%%%%%%%%%%%%
\begin{remark}\label{C.0.499}
We remark that Corollary \ref{C.0.5} is an extension of \cite[Theorem 1]{K.2005}.
\end{remark}
%%%%%%%%%%%%%%%%%%%%%%%%%%%%%%%%%%
\begin{remark}\label{R.0.40}
Since $x^{\sharp_{a}}x$ is an $a$-positive element in $\mathfrak{A}$, we have
\begin{align*}
{\big\|xx^{\sharp_{a}} + x^{\sharp_{a}}x\big\|}_{a}\geq {\big\|xx^{\sharp_{a}}\big\|}_{a} = {\|x\|}^2_{a}.
\end{align*}
Hence
\begin{align*}
\frac{1}{4}{\|x\|}^2_{a} \leq \frac{{\big\|xx^{\sharp_{a}} + x^{\sharp_{a}}x\big\|}_{a}}{4}
+ \frac{1}{8}\displaystyle{\sup_{\theta \in \mathbb{R}}}
\Big|{\big\|e^{i\theta}x + e^{-i\theta}x^{\sharp_{a}}\big\|}^2_{a} - {\big\|e^{i\theta}x - e^{-i\theta}x^{\sharp_{a}}\big\|}^2_{a}\Big|.
\end{align*}
Thus the inequality \eqref{I.2.C.0.5} refines the inequality $\frac{1}{2}{\|x\|}_{a} \leq v_{a}(x)$.
\end{remark}
%%%%%%%%%%%%%%%%%%%%%%
As an immediate consequence of Theorem \ref{C.00.3} and Corollary \ref{C.0.5}, we have the following result.
%%%%%%%%%%%%%%%%%%%%%%%%%%
\begin{corollary}\label{C.0.59}
Let $\mathfrak{A}$ be a $C^*$-algebra and let $x\in{\mathfrak{A}}_{a}$. If $ax^2 = 0$, then $v^2_{a}(x)=\frac{1}{4}{\big\|xx^{\sharp_{a}} + x^{\sharp_{a}}x\big\|}_{a}$.
\end{corollary}
%%%%%%%%%%%%%%%%%%%%%
\begin{remark}\label{R.0.509}
Let $\mathfrak{A}$ be a $C^*$-algebra and let $x\in{\mathfrak{A}}_{a}$. Suppose that $ax\neq 0$ and $ax^2 = 0$.
By the Gelfand--Naimark theorem, $\mathfrak{A}$ can be considered as a norm
closed $*$-subalgebra of $\mathbb{B}(\mathcal{H})$ for some Hilbert space $\big(\mathcal{H}, \langle \cdot, \cdot\rangle\big)$.
By \cite[Theorem~4.1]{B.M.Positivity}, we know that
$v_{a}(x)=\sup\big\{\big|\langle ax\xi, \xi\rangle\big|: \xi\in \mathcal{H}, \, {\|\xi\|}_{a} = 1 \big\}$
and
${\|x\|}^2_{a}=\sup\big\{\big|\langle x^*ax\xi, \xi\rangle\big|: \xi\in \mathcal{H}, \, {\|\xi\|}_{a} = 1 \big\}$.
Hence by \cite[Corollary~2]{FekiAn2020} we have $v_{a}(x)=\frac{1}{2}{\|x\|}_{a}$.
\end{remark}
%%%%%%%%%%%%%%%%%%%%%%
Next we prove the following inequality which improves on the lower bound the inequality $\frac{1}{2}{\|x\|}_{a} \leq v_{a}(x)$.
%%%%%%%%%%%%%%%%%%%%%%%%%%%
\begin{theorem}\label{T.51.3}
Let $\mathfrak{A}$ be a $C^*$-algebra and let $x\in{\mathfrak{A}}_{a}$. Then
\begin{align*}
\frac{1}{2}{\|x\|}_{a} + \frac{1}{4}\Big|{\big\|x + ix^{\sharp_{a}}\big\|}_{a} - {\big\|x -ix^{\sharp_{a}}\big\|}_{a}\Big|\leq v_{a}(x).
\end{align*}
\end{theorem}
\begin{proof}
Let $f\in {\mathcal{S}}(\mathfrak{A})$ with $f(a)\neq 0$. We have
\begingroup\makeatletter\def\f@size{10}\check@mathfonts
\begin{align*}
\frac{1}{\sqrt{2}}\left|f\left(a\left(\frac{x+x^{\sharp_{a}}}{2}\pm\frac{x-x^{\sharp_{a}}}{2i}\right)\right)\right|
&= \frac{1}{\sqrt{2}}\left|f\left(\frac{ax+x^*a}{2}\pm\frac{ax-x^*a}{2i}\right)\right|
\\&\leq \frac{1}{\sqrt{2}}\left(\left|f\left(\frac{ax+x^*a}{2}\right)\right| + \left|f\left(\frac{ax-x^*a}{2i}\right)\right|\right)
\\& \leq \sqrt{f^2\left(\frac{ax+x^*a}{2}\right) + f^2\left(\frac{ax-x^*a}{2i}\right)}
\\& = \sqrt{\left(\frac{f(ax)+ \overline{f}(ax)}{2}\right)^2 + \left(\frac{f(ax)- \overline{f}(ax)}{2i}\right)^2}
\\& = |f(ax)| \leq f(a)v_{a}(x).
\end{align*}
\endgroup
Hence
\begin{align}\label{F.1.T.51.3}
\frac{1}{\sqrt{2}}\frac{\left|f\left(a\left(\frac{x+x^{\sharp_{a}}}{2}\pm\frac{x-x^{\sharp_{a}}}{2i}\right)\right)\right|}{f(a)}
\leq v_{a}(x).
\end{align}
Since $\frac{x+x^{\sharp_{a}}}{2}\pm\frac{x-x^{\sharp_{a}}}{2i}$ are $a$-self-adjoint,
by taking the supremum over $f\in {\mathcal{S}}(\mathfrak{A})$ with $f(a)\neq 0$ in \eqref{F.1.T.51.3}, we get
\begin{align*}
\frac{1}{\sqrt{2}}{\left\|\frac{x+x^{\sharp_{a}}}{2}\pm\frac{x-x^{\sharp_{a}}}{2i}\right\|}_{a} \leq v_{a}(x).
\end{align*}
Thus
\begin{align*}
\frac{1}{\sqrt{2}}\max\left\{{\left\|\frac{x+x^{\sharp_{a}}}{2}+\frac{x-x^{\sharp_{a}}}{2i}\right\|}_{a},
{\left\|\frac{x+x^{\sharp_{a}}}{2}-\frac{x-x^{\sharp_{a}}}{2i}\right\|}_{a}\right\} \leq v_{a}(x),
\end{align*}
or equivalently,
\begin{align*}
\frac{1}{2}\max\Big\{{\|x+ix^{\sharp_{a}}\|}_{a},
{\|x-ix^{\sharp_{a}}\|}_{a}\Big\} \leq v_{a}(x).
\end{align*}
Therefore,
\begin{align*}
v_{a}(x)&\geq \frac{{\|x+ix^{\sharp_{a}}\|}_{a}+{\|x-ix^{\sharp_{a}}\|}_{a}}{4} + \frac{\Big|{\|x+ix^{\sharp_{a}}\|}_{a}-{\|x-ix^{\sharp_{a}}\|}_{a}\Big|}{4}
\\& \geq \frac{{\Big\|(x+ix^{\sharp_{a}})+(x-ix^{\sharp_{a}})\Big\|}_{a}}{4} + \frac{\Big|{\|x+ix^{\sharp_{a}}\|}_{a}-{\|x-ix^{\sharp_{a}}\|}_{a}\Big|}{4}
\\& = \frac{1}{2}{\|x\|}_{a} + \frac{1}{4}\Big|{\big\|x + ix^{\sharp_{a}}\big\|}_{a} - {\big\|x -ix^{\sharp_{a}}\big\|}_{a}\Big|.
\end{align*}
\end{proof}
%%%%%%%%%%%%%%%%%%%%%%
For $x\in{\mathfrak{A}}_{a}$, following \cite{A.K.1},
let $d_{_{a}}(x)$ denotes the $a$-numerical radius distance
of $x$ from the scalar elements, that is,
\begin{align*}
d_{_{a}}(x) = \displaystyle{\inf_{\zeta \in \mathbb{C}}} \,v_{_{a}}(x - \zeta \textbf{1}).
\end{align*}
Then there exists a complex number $\zeta_0$ satisfying
$d_{_{a}}(x) = v_{_{a}}(x - \zeta_0 \textbf{1})$. Indeed,
put $\mathbb{D} := \left\{\zeta \in \mathbb{C}: \, |\zeta| \leq 2v_{_{a}}(x)\right\}$
and define $f:\, \mathbb{D}\rightarrow  \mathbb{R}^{\geq 0}$
by the formula $f(\zeta) = v_{_{a}}\big(x - \zeta \textbf{1}\big)$.
Clearly, $f$ is continuous and attains its minimum at, say, $\zeta_0 \in \mathbb{D}$ (of course, there may be
many such points). Then $v_{_{a}}\big(x - \zeta \textbf{1}\big) \geq v_{_{a}}\big(x - \zeta_0 \textbf{1}\big)$
for all $\zeta \in \mathbb{D}$. In particular, in view of $0 \in \mathbb{D}$,
we have $v_{_{a}}(x) \geq v_{_{a}}\big(x - \zeta_0 \textbf{1}\big)$.
If $\zeta \notin \mathbb{D}$, then $|\zeta| > 2v_{_{a}}(x)$. Thus
\begin{align*}
v_{_{a}}\big(x - \zeta \textbf{1}\big) \geq |\zeta| - v_{_{a}}(x)
> 2v_{_{a}}(x) - v_{_{a}}(x) = v_{_{a}}(x) \geq v_{_{a}}\big(x - \zeta_0 \textbf{1}\big),
\end{align*}
whence $v_{_{a}}\big(x - \zeta \textbf{1}\big) \geq v_{_{a}}\big(x - \zeta_0 \textbf{1}\big)$
for all $\zeta \notin \mathbb{D}$.
Therefore, $v_{_{a}}\big(x - \zeta\textbf{1}\big) \geq v_{_{a}}\big(x - \zeta_0 \textbf{1}\big)$
for all $\zeta \in \mathbb{C}$ and so $d_{_{a}}(x) = v_{_{a}}\big(x - \zeta_0 \textbf{1}\big)$.
%%%%%%%%%%%%%%%%%%%%%%%%

Next, we obtain a more precise estimate than the first inequality in Corollary \ref{C.0.5}.
%%%%%%%%%%%%%%%%%%%%%%%%%%%%
\begin{theorem}\label{t.2.13.9}
Let $\mathfrak{A}$ be a $C^*$-algebra and let $x\in {\mathfrak{A}}_{a}$. Then
\begin{align*}
\frac{1}{4}{\left\|xx^{\sharp_{a}} + x^{\sharp_{a}}x\right\|}_{a} \leq \frac{1}{2}\Big(v_{_{a}}^2(x) + d_{_{a}}^2(x)\Big)
\leq v_{_{a}}^2(x).
\end{align*}
\end{theorem}
\begin{proof}
Clearly,
$\frac{1}{2}\Big(v_{_{a}}^2(x) + d_{_{a}}^2(x)\Big)
\leq v_{_{a}}^2(x)$.
It is therefore enough to prove the first inequality.
Let $\zeta_0 \in \mathbb{C}$ such that
\begin{align*}
d_{_{a}}(x) = v_{_{a}}(x - \zeta_0\textbf{1}).
\end{align*}
If $\zeta_0 = 0$, then $d_{_{a}}(x) = v_{_{a}}(x)$. By employing \eqref{F.2.T.5.3} with $t=s=\frac{1}{2}$ and $\theta = 0$,
the triangle inequality for the semi-norm ${\|\!\cdot\!\|}_{a}$ and Remark \ref{R.2.2} we have
\begin{align*}
\frac{1}{4}{\left\|xx^{\sharp_{a}} + x^{\sharp_{a}}x\right\|}_{a} &
= \frac{1}{2}{\left\|\mathfrak{R}^2(x) + \mathfrak{I}^2(x)\right\|}_{a}
\\& \leq \frac{1}{2}\Big({\left\|\mathfrak{R}^2(x)\right\|}_{a}
+ {\left\|\mathfrak{I}^2(x)\right\|}_{a}\Big)
\\& \leq \frac{1}{2}\Big({\left\|\mathfrak{R}(x)\right\|}^2_{a}
+ {\left\|\mathfrak{I}(x)\right\|}^2_{a}\Big)
\\& \leq \frac{1}{2}\Big(v_{_{a}}^2(x) + v_{_{a}}^2(x)\Big)
= \frac{1}{2}\Big(v_{_{a}}^2(x) + d_{_{a}}^2(x)\Big).
\end{align*}
If $\zeta_0 \neq 0$, then put $\zeta = \frac{\overline{\zeta_0}}{|\zeta_0|}$.
As above we get that
\begin{align}\label{F.2.T.5.38}
\frac{1}{4}{\left\|(\zeta x)(\zeta x)^{\sharp_{a}} + (\zeta x)^{\sharp_{a}}(\zeta x)\right\|}_{a}
\leq \frac{1}{2}\Big({\left\|\mathfrak{R}(\zeta x)\right\|}^2_{a}
+ {\left\|\mathfrak{I}(\zeta x)\right\|}^2_{a}\Big).
\end{align}
Since $xx^{\sharp_{a}} + x^{\sharp_{a}}x=(\zeta x)(\zeta x)^{\sharp_{a}} + (\zeta x)^{\sharp_{a}}(\zeta x)$ and
$\mathfrak{I}\big(\zeta (x-\zeta_0\textbf{1})\big) = \mathfrak{I}(\zeta x)$, by \eqref{F.2.T.5.38}, we have
\begin{align*}
\frac{1}{4}{\left\|xx^{\sharp_{a}} + x^{\sharp_{a}}x\right\|}_{a} &
\leq \frac{1}{2}\Big({\left\|\mathfrak{R}(\zeta x)\right\|}^2_{a}
+ {\left\|\mathfrak{I}\big(\zeta (x-\zeta_0\textbf{1})\big)\right\|}^2_{a}\Big)
\\& \leq \frac{1}{2}\Big(v_{_{a}}^2(\zeta x) + v_{_{a}}^2(\zeta (x-\zeta_0\textbf{1})\Big)
= \frac{1}{2}\Big(v_{_{a}}^2(x) + d_{_{a}}^2(x)\Big).
\end{align*}
This completes the proof.
\end{proof}
%%%%%%%%%%%%%%%%%%%%%%%%%
We finish this section by an upper bound
for the $a$-numerical radius of products of elements in $C^*$-algebras.
%%%%%%%%%%%%%%%%%%%%%%%%
\begin{theorem}\label{T.6.289}
Let $\mathfrak{A}$ be a $C^*$-algebra and let $x, y\in {\mathfrak{A}}_{a}$. The following inequalities hold:
\begin{align*}
v_{_{a}}(xy) \leq {\|xy\|}_{a} \leq \min\big\{K_1, K_2, K_3\big\} \leq 4v_{_{a}}(x)v_{_{a}}(y),
\end{align*}
where $K_1= {\|x\|}_{a}\big(v_{_{a}}(y) + d_{_{a}}(y)\big)$,
$K_2 = {\|y\|}_{a}\big(v_{_{a}}(x) + d_{_{a}}(x)\big)$
and $K_3 = \big(v_{_{a}}(x) + d_{_{a}}(x)\big)\big(v_{_{a}}(y) + d_{_{a}}(y)\big)$.
\end{theorem}
\begin{proof}
It is evident that $v_{_{a}}(xy) \leq {\|xy\|}_{a}$.
The first inequality in \eqref{I.000009} and the fact that $d_{_{a}}(z)\leq v_{_{a}}(z)$ holds for every $z\in{\mathfrak{A}}_{a}$
imply the third desired inequality.
It is therefore enough to prove the second inequality.
Let $\zeta_0, \xi_0 \in \mathbb{C}$ such that $d_{_{a}}(x) = v_{_{a}}(x - \zeta_0\textbf{1})$ and $d_{_{a}}(y) = v_{_{a}}(y - \xi_0\textbf{1})$.
If $\xi_0 = 0$, then by the first inequality in \eqref{I.000009} we get
\begin{align*}
{\|xy\|}_{a} \leq {\|x\|}_{a}{\|y\|}_{a} \leq 2{\|x\|}_{a} v_{_{a}}(y) = {\|x\|}_{a}\big(v_{_{a}}(y) + d_{_{a}}(y)\big) = K_1.
\end{align*}
Hence, we may assume that $\xi_0 \neq 0$. Put $\xi = \frac{\overline{\xi_0}}{|\xi_0|}$.
Then $\mathfrak{I}\big(\xi (y-\xi_0\textbf{1}\big) = \mathfrak{I}(\xi y)$ and by Remark \ref{R.2.2} we have
\begin{align*}
{\|xy\|}_{a} &= {\left\|x(\xi y)\right\|}_{a}
\\& \leq {\|x\|}_{a}{\|\xi y\|}_{a}
\\&= {\|x\|}_{a}{\|\mathfrak{R}(\xi y) + i\mathfrak{I}(\xi y)\|}_{a}
\\& \leq {\|x\|}_{a}\Big({\|\mathfrak{R}(\xi y)\|}_{a} + {\|\mathfrak{I}(\xi y)\|}_{a}\Big)
\\& = {\|x\|}_{a}\Big({\|\mathfrak{R}(\xi y)\|}_{a} + {\|\mathfrak{I}(\xi (y-\xi_0\textbf{1})\|}_{a}\Big)
\\& \leq {\|x\|}_{a}\Big(v_{_{a}}(\xi y) + v_{_{a}}(\xi (y-\xi_0\textbf{1})\Big)
= {\|x\|}_{a}\big(v_{_{a}}(y) + d_{_{a}}(y)\big).
\end{align*}
Thus
\begin{align}\label{I.1.T.6.2}
{\|xy\|}_{a} \leq K_1.
\end{align}
By a similar argument, we obtain
\begin{align}\label{I.2.T.6.2}
{\|xy\|}_{a} \leq K_2.
\end{align}
Further, put $\zeta = \frac{\overline{\zeta_0}}{|\zeta_0|}$.
Then $\mathfrak{I}(\zeta (x-\zeta_0\textbf{1}) = \mathfrak{I}(\zeta x)$.
We have
\begin{align*}
{\|xy\|}_{a} &= {\|(\zeta x)(\xi y)\|}_{a}
\\ &\leq {\|\zeta x\|}_{a}{\|\xi y\|}_{a}
\\& = {\left\|\mathfrak{R}(\zeta x) + i\mathfrak{I}(\zeta x)\right\|}_{a}
{\left\|\mathfrak{R}(\xi y) + i\mathfrak{I}(\xi y)\right\|}_{a}
\\& \leq \Big({\left\|\mathfrak{R}(\zeta x)\right\|}_{a} + {\left\|\mathfrak{I}(\zeta x)\right\|}_{a}\Big)
\Big({\left\|\mathfrak{R}(\xi y)\right\|}_{a} + {\left\|\mathfrak{I}(\xi y)\right\|}_{a}\Big)
\\& = \Big({\left\|\mathfrak{R}(\zeta x)\right\|}_{a} + {\left\|\mathfrak{I}(\zeta (x-\zeta_0\textbf{1})\right\|}_{a}\Big)
\Big({\left\|\mathfrak{R}(\xi y)\right\|}_{a} + {\left\|\mathfrak{I}(\xi (y-\xi_0\textbf{1})\right\|}_{a}\Big)
\\& \leq \Big(v_{_{a}}(\zeta x) + v_{_{a}}\big(\zeta (x - \zeta_0\textbf{1})\big)\Big)\Big(v_{_{a}}(\xi y) + v_{_{a}}\big(\xi (y - \xi_0\textbf{1})\big)\Big)
\\& = \big(v_{_{a}}(x) + d_{_{a}}(x)\big)\big(v_{_{a}}(y) + d_{_{a}}(y)\big),
\end{align*}
and so
\begin{align}\label{I.3.T.6.2}
{\|xy\|}_{a} \leq K_3.
\end{align}
Applying \eqref{I.1.T.6.2}, \eqref{I.2.T.6.2} and \eqref{I.3.T.6.2}, we arrive at
\begin{align*}
{\|xy\|}_{a} \leq \min\big\{K_1, K_2, K_3\big\}.
\end{align*}
\end{proof}
%%%%%%%%%%%%%%%%%%%
%%%%%%%%%%%%%%%%%%%%
\section{Additional results}
Recall that a character on a commutative $C^*$-algebra $\mathfrak{A}$
is a non-zero homomorphism $\varphi: \mathfrak{A} \rightarrow \mathbb{C}$.
We denote by $\widehat{\mathfrak{A}}$ the set of characters on $\mathfrak{A}$.
Following \cite{bakloutiNamouri, FekiAn2020}, if $x$ is an element of ${\mathfrak{A}}_{a}$, its $a$-spectral radius is defined to be
\begin{align*}
r_{a}(x) = \displaystyle{\lim_{n\rightarrow +\infty}}{\|x^{n}\|}^{1/n}_{a}.
\end{align*}
The study of the spectral radius of operators received considerable
attention in the last decades. The reader may consult \cite{Kula.Wysoczanski, Mu, Singh} and the references therein.
We first prove two lemmas that we need in what follows.
%%%%%%%%%%%%%%%%%%%%%%%%
\begin{lemma}\label{L.04.18}
Let $\mathfrak{A}$ be a commutative $C^*$-algebra and let $x\in{\mathfrak{A}}_{a}$. Then
\begin{align*}
{\|x\|}_{a} = v_{a}(x) = \sup\big\{|\varphi(x)|:\, \varphi\in \widehat{\mathfrak{A}}, \varphi(a) \neq 0\big\}.
\end{align*}
In particular $r_{a}(x) = {\|x\|}_{a} = {\|x\|}_{a^{1/2}}$ and $v_{a}(x) = v_{a^{1/2}}(x)$ for any $x\in \mathfrak{A}^a$.
\end{lemma}
\begin{proof}
Set $\beta=\sup\big\{|\varphi(x)|:\, \varphi\in \widehat{\mathfrak{A}}, \varphi(a) \neq 0\big\}$.
It is clear that $\beta\leq v_{a}(x)\leq {\|x\|}_{a}$.
Pick up an element $f$ in ${\mathcal{S}}(\mathfrak{A})$.
By Theorem~5.1.6 and Corollary~5.1.10 of \cite{Mu} there exists a net
$(f_\alpha)_{\alpha\in\Lambda}\in{\rm conv}(\widehat{\mathfrak{A}})$ that weakly converges to $f$.
For each $\alpha\in\Lambda$, there are a positive integer $n_\alpha$, positive scalars $(\lambda_k)_{1\leq k\leq n_\alpha}$
and characters $(\varphi_k)_{1\leq k\leq n}\subseteq\widehat{\mathfrak{A}}$ such that $\sum_{k=1}^n\lambda_k =1$
and $f_\alpha=\sum_{k=1}^n\lambda_k \varphi_k$.
For any $\alpha\in\Lambda$, we have
\begin{align*}
f_\alpha(x^*ax)=\sum_{k=1}^n\lambda_k \varphi_k(x^*ax)
= \sum_{k=1}^n\lambda_k \varphi_k(a)|\varphi_k(x)|^2 \leq \beta^2 \sum_{k=1}^n\lambda_k \varphi_k(a) = \beta^2f_\alpha(a).
\end{align*}
Thus
\begin{align*}
f_\alpha(x^*ax)\leq \beta^2f_\alpha(a).
\end{align*}
Taking the limit in the above inequality, yields that $f(x^*ax)\leq \beta^2\ f(a)$ for all $f\in{\mathcal{S}}(\mathfrak{A})$.
This shows that ${\|x\|}_{a}\leq \beta$, and ends the proof.
\end{proof}
%%%%%%%%%%%%%%%%%%%%%%%%%%%%%%%%%%
We shall need also the following elementary observation about the $a$-spectral radius.
%%%%%%%%%%%%%%%%%%%%%%%
\begin{lemma}\label{L.04.181}
Let $\mathfrak{A}$ be a $C^*$-algebra and let $x, y\in{\mathfrak{A}}_{a}$. Then $r_{a}(xy)=r_{a}(yx)$.
\end{lemma}
\begin{proof}
For any $n \in \mathbb{N}$, we have
\begin{align*}
{\big\|(xy)^{n}\big\|}^{1/n}_{a} = {\big\|x(yx)^{n-1}y\big\|}^{1/n}_{a} \leq {\|x\|}^{1/n}_{a}{\big\|(yx)^{n-1}\big\|}^{1/n}_{a}{\|y\|}^{1/n}_{a}.
\end{align*}
Thus
\begin{align*}
{\big\|(xy)^{n}\big\|}^{1/n}_{a} \leq {\|x\|}^{1/n}_{a}\Big({\big\|(yx)^{n-1}\big\|}^{1/{n-1}}_{a}\Big)^{(n-1)/n}{\|y\|}^{1/n}_{a}.
\end{align*}
Taking the limit when $n$ tends to $\infty$ in the above inequality, we infer that $r_{a}(xy)\leq r_{a}(yx)$.
In a similar way we can prove that $r_{a}(xy)\geq r_{a}(yx)$.
\end{proof}
%%%%%%%%%%%%%%%%%%%%%%%%%%%%%%%%
The next result provides a version of \cite[Lemma~3.3]{D-C} for the seminorm ${\|\!\cdot\!\|}_{a}$.
%%%%%%%%%%%%%%%%%%%%%%%%%%%%%%%
\begin{theorem}\label{T.04.01}
Let $\mathfrak{A}$ be a $C^*$-algebra and let $x$ and $y$ be $a$-self-adjoint elements in $\mathfrak{A}$
such that ${\|y\|}_{a}\leq {\|x\|}_{a} \leq 1$. Then ${\|x+y\|}_{a}\leq 1+2{\|xy\|}_{a}$.
\end{theorem}
\begin{proof}
Firstly, note that $x$ and $y$ are in $\mathfrak{A}_{a}$.
Now, let $\kappa$ be the smallest number such that ${\|b+c\|}_{a}\leq 1+\kappa$
for all $a$-self-adjoint elements $b$ and $c$ in $\mathfrak{A}$, where ${\|c\|}_{a}\leq {\|b\|}_{a} \leq 1$ and ${\|bc\|}_{a}\leq {\|xy\|}_{a}$.
We claim that $\kappa\leq 2{\|xy\|}_{a}$. Suppose the contrary that $\kappa> 2{\|xy\|}_{a}$.
Note that $\kappa$ exists since $x, y\in\mathfrak{A}_{a}$.
Pick two $a$-self-adjoint elements $b$ and $c$ in $\mathfrak{A}$.
Then by \eqref{I.00000} and the triangle inequality, we have
\begin{align}\label{T.I.1.04.01}
{\|b+c\|}^2_{a} = {\|(b+c)^2\|}_{a} \leq {\|b^2+c^2\|}_{a} + 2{\|bc\|}_{a}.
\end{align}
It is easy to see that $b^2$ and $c^2$ are $a$-self-adjoint and so by \eqref{I.00000} we have
\begin{align*}
{\|c^2\|}_{a} = {\|c\|}^2_{a} \leq {\|b\|}^2_{a} = {\|b^2\|}_{a} \leq 1,
\end{align*}
and
\begin{align*}
{\|b^2c^2\|}_{a} = {\|b(bc)c\|}_{a} \leq  {\|b\|}_{a}{\|bc\|}_{a}{\|c\|}_{a} \leq {\|bc\|}_{a} \leq {\|xy\|}_{a}.
\end{align*}
Therefore, ${\|b^2+c^2\|}_{a}\leq 1+\kappa$
and by \eqref{T.I.1.04.01} it follows that
\begin{align*}
{\|b + c\|}_{a} \leq \sqrt{1 + 2\kappa} < 1 + \kappa,
\end{align*}
for all $a$-self-adjoint elements $b$ and $c$ in $\mathfrak{A}$ such that ${\|c\|}_{a}\leq {\|b\|}_{a} \leq 1$ and ${\|bc\|}_{a}\leq {\|xy\|}_{a}$.
This is a contradiction and hence $\kappa\leq 2{\|xy\|}_{a}$ as claimed.
\end{proof}
%%%%%%%%%%%%%%%%%%%%%%%%%%%
As an immediate consequence of Theorem \ref{T.04.01}, we get the following result.
%%%%%%%%%%%%%%%%%%%%%%%%%
\begin{corollary}\label{C.04.02}
Let $\mathfrak{A}$ be a $C^*$-algebra and let $x$ and $y$ be $a$-self-adjoint elements in $\mathfrak{A}$
such that ${\|y\|}_{a}\leq {\|x\|}_{a}$ and ${\|x\|}_{a} \neq 0$. Then ${\|x+y\|}_{a}\leq {\|x\|}_{a}+2\frac{{\|xy\|}_{a}}{{\|x\|}_{a}}$.
\end{corollary}
%%%%%%%%%%%%%%%%%%%%%%%%
\begin{remark}\label{R.04.03}
In fact, there is a smallest number $\kappa \in [0, 2]$ such that
\begin{align}\label{R.I.1.04.03}
{\|x+y\|}_{a}\leq {\|x\|}_{a}+\kappa\frac{{\|xy\|}_{a}}{{\|x\|}_{a}}
\end{align}
holds whenever $x$ and $y$ are two $a$-self-adjoint elements in $\mathfrak{A}$ with ${\|x\|}_{a} \neq 0$ and ${\|y\|}_{a}\leq {\|x\|}_{a}$.
\end{remark}
%%%%%%%%%%%%%%%%%%%%%%%%%%%
If $\mathfrak{A}$ is commutative, then $\mathfrak{A}_{a}=\mathfrak{A}^{a}=\mathfrak{A}$ and $\kappa$ may be chosen equal to $1$.
%%%%%%%%%%%%%%%%%%%%%%%%%%%%%%%%
\begin{theorem}\label{T.04.19}
Let $\mathfrak{A}$ be a commutative $C^*$-algebra and let $x$ and $y$ be $a$-self-adjoint elements in $\mathfrak{A}$
such that ${\|y\|}_{a}\leq {\|x\|}_{a}$ and ${\|x\|}_{a} \neq 0$. Then
\begin{align*}
{\|x+y\|}_{a}\leq {\|x\|}_{a}+\frac{{\|xy\|}_{a}}{{\|x\|}_{a}}.
\end{align*}
\end{theorem}
\begin{proof}
By Lemma \ref{L.04.18}, we know that ${\|\!\cdot\!\|}_{a}={\|\!\cdot\!\|}_{a^{1/2}}$.
So if $x, y\in\mathfrak{A}$ are two $a$-self-adjoint elements such that ${\|y\|}_{a}\leq {\|x\|}_{a}$ and ${\|x\|}_{a} \neq 0$,
then $\big({\|x\|}_{a}\textbf{1}-x\big)\big({\|x\|}_{a}\textbf{1}-y\big)$ and $\big({\|x\|}_{a}\textbf{1}+x\big)\big({\|x\|}_{a}\textbf{1}+y\big)$
are also $a$-positive. Hence for every $\varphi\in\widehat{\mathfrak{A}}$ with $\varphi(a)\neq0$ we have
\begin{align*}
\varphi\Big(\big({\|x\|}_{a}\textbf{1}-x\big)\big({\|x\|}_{a}\textbf{1}-y\big)\Big)
=\frac{\varphi\Big(a\big({\|x\|}_{a}\textbf{1}-x\big)\big({\|x\|}_{a}\textbf{1}-y\big)\Big)}{\varphi(a)}\geq 0
\end{align*}
and
\begin{align*}
\varphi\Big(\big({\|x\|}_{a}\textbf{1}-x\big)\big({\|x\|}_{a}\textbf{1}-y\big)\Big)
=\frac{\varphi\Big(a\big({\|x\|}_{a}\textbf{1}-x\big)\big({\|x\|}_{a}\textbf{1}-y\big)\Big)}{\varphi(a)}\geq 0.
\end{align*}
Accordingly
\begin{align}\label{T.I.1.04.19}
{\|x\|}^2_{a} + \varphi(xy) \geq {\|x\|}_{a}\varphi(x+y) \quad \mbox{and} \quad {\|x\|}^2_{a} + \varphi(xy) \geq -{\|x\|}_{a}\varphi(x+y),
\end{align}
for any $\varphi\in\widehat{\mathfrak{A}}$ with $\varphi(a)\neq0$. Now, by Lemma \ref{L.04.18}, we see that
\begin{align*}
{\|x\|}^2_{a} + \varphi(xy) \leq {\|x\|}^2_{a} + {\|xy\|}_{a}.
\end{align*}
The above inequality together with \eqref{T.I.1.04.19} yields that,
\begin{align*}
{\|x\|}_{a} \big|\varphi(x+y)\big| \leq {\|x\|}^2_{a} + {\|xy\|}_{a},
\end{align*}
for any $\varphi\in\widehat{\mathfrak{A}}$ with $\varphi(a)\neq0$.
Taking the suppremum over all $\varphi\in\widehat{\mathfrak{A}}$ with $\varphi(a)\neq0$
and dividing by ${\|x\|}_{a}$ we get ${\|x+y\|}_{a}\leq {\|x\|}_{a}+\frac{{\|xy\|}_{a}}{{\|x\|}_{a}}$. This ends the proof.
\end{proof}
%%%%%%%%%%%%%%%%%%%%%%%
The next theorem gives a necessary condition for which the inequality
\eqref{R.I.1.04.03} holds for all elements of $\mathfrak{A}_{a}$.
%%%%%%%%%%%%%%%%%%%%%%%
\begin{theorem}\label{T.04.20}
Let $\mathfrak{A}$ be a $C^*$-algebra. If the inequality \eqref{R.I.1.04.03}
holds for any $x, y\in\mathfrak{A}_{a}$, then $axy=ayx$ for all $x, y\in\mathfrak{A}_{a}$.
In particular, if the map $L_{a}: \mathfrak{A}\rightarrow \mathfrak{A}$ given by $L_{a}(z)=az$ is injective
then $\mathfrak{A}_{a}=\mathfrak{A}$ and $\mathfrak{A}$ is commutative.
\end{theorem}
\begin{proof}
Firstly, note that if $z\in\mathfrak{A}_{a}$ then by the proof of \cite[Proposition~3.3]{B.M.Positivity}
$f(z^*az)\leq {\|z\|}^2_{a} f(a)$ for any $f\in{\mathcal{S}}(\mathfrak{A})$.
Accordingly
\begin{align}\label{T.I.1.04.20}
\big\|a^{1/2}z\big\|\leq {\|z\|}_{a}\big\|a^{1/2}\big\|, \quad (z\in\mathfrak{A}_{a}).
\end{align}
Now, assume that the inequality \eqref{R.I.1.04.03}
holds for any $x, y\in\mathfrak{A}_{a}$.
In particular, by setting $y=x$ in \eqref{R.I.1.04.03}, we obtain ${\|x\|}^2_{a}\leq \kappa{\|x^2\|}_{a}$,
and so by induction ${\|x\|}^{2^n}_{a}\leq \kappa^{2^n-1}{\|x^{2^n}\|}_{a}$.
Therefore
\begin{align}\label{T.I.2.04.20}
{\|x\|}_{a} \leq \kappa\,r{_a}(x).
\end{align}
For any complex $\lambda$, set $h(\lambda):=e^{\lambda x}ye^{-\lambda x}$.
Note that if $z\in\mathfrak{A}_{a}$ then $e^{\lambda z}\in\mathfrak{A}_{a}$ for every $\lambda \in \mathbb{C}$.
This is because ${\|e^{\lambda z}\|}_{a}\leq e^{{\|\lambda z\|}_{a}}$.
Keeping in mind inequalities \eqref{T.I.1.04.20} and \eqref{T.I.2.04.20}, it yields by Lemma \ref{L.04.181} that
\begin{align*}
\big\|a^{1/2}h(\lambda)\big\|\leq \kappa\,r{_a}(y)\big\|a^{1/2}\big\|,
\end{align*}
for any complex $\lambda$.
By Liouville's theorem (see \cite[Theorem~3]{Singh}) it yields $ae^{\lambda x} y=aye^{\lambda x}$.
By expanding the first few terms of the above identity, we infer that $axy=ayx$.
\end{proof}
%%%%%%%%%%%%%%%%%%%%%%%%%%%
Analogously to the usual numerical index (see \cite[pp.~43-44]{bonsall_duncan_1973}),
we define $a$-numerical index of $\mathfrak{A}$ by the number
\begin{align*}
n_{a}(\mathfrak{A}) = \inf\big\{v_{a}(x): \,x\in\mathfrak{A}_{a}, {\|x\|}_{a}= 1\big\}.
\end{align*}
It is clear from \eqref{I.000009} that $\frac{1}{2}\leq n_{a}(\mathfrak{A})\leq 1$.
We can state the following.
%%%%%%%%%%%%%%%%%%%%%
\begin{theorem}\label{T.04.21}
Let $\mathfrak{A}$ be a $C^*$-algebra. The following statements hold.
\begin{itemize}
\item[(i)] If the algebra $\mathfrak{A}_{a}$ is commutative then $n_{a}(\mathfrak{A})=1$.
\item[(ii)] If $\mathfrak{A}$ is not commutative and $a$ is invertible then $n_{a}(\mathfrak{A}) = \frac{1}{2}$.
\item[(iii)] If there exists $x\in\mathfrak{A}_{a}$
such that $ax\neq 0$ and $ax^2=0$ then $n_{a}(\mathfrak{A}) = \frac{1}{2}$.
\end{itemize}
\end{theorem}
\begin{proof}
(i) It follows immediately from Lemma \ref{L.04.18}.

(ii) If $\mathfrak{A}$ is not commutative and $a$ is invertible, then by
\cite[Theorem~3]{crabb1974} and the fact that
$v_{a}(z)=v\big(a^{1/2}z(a^{1/2})^{-1}\big)$ for any $z\in\mathfrak{A}$, we have $n_{a}(\mathfrak{A}) = \frac{1}{2}$.

(iii) If there exists an element $x\in\mathfrak{A}_{a}$ such that $ax\neq 0$ and $ax^2=0$,
then by Remark \ref{R.0.509} $v_{a}(x)=\frac{1}{2}{\|x\|}_{a}$. Hence $n_{a}(\mathfrak{A}) = \frac{1}{2}$.
\end{proof}
%%%%%%%%%%%%%%%%%%%%%%%%%
We close this paper by the following remark.
%%%%%%%%%%%%%%%%%%%%%%%
\begin{remark}\label{R.04.22}
It well know that a $C^*$-algebra $\mathfrak{A}$ is noncommutative if and only if there exists
a nonzero element $x\in\mathfrak{A}$ such that $x^2=0$.
Such an element need not to be in $\mathfrak{A}_{a}$ or it satisfy $au=0$ in general.
To see why this take $\mathfrak{A}=\mathcal{M}_2(\mathbb{C})$, the $C^*$-algebra of all complex $2\times 2$ matrices, and let
\begin{align*}
A =\begin{bmatrix}
1 & 0\\
0 & 0
\end{bmatrix}, \, X =\begin{bmatrix}
0 & 1\\
0 & 0
\end{bmatrix}, \, Y =\begin{bmatrix}
0 & 0\\
1 & 0
\end{bmatrix}.
\end{align*}
Easy computation shows that  $X^2=Y^2=AY=0$, but $X\notin\mathfrak{A}_A$ since $v_{A}(X)=\infty$.
Also, it is well known for the numerical index, that a $C^*$-algebra $\mathfrak{A}$ is commutative
or not commutative according as $n(\mathfrak{A})$ is $1$ or $\frac{1}{2}$, see \cite[Theorem 3]{crabb1974}.
Unfortunately, this fails to be true for the $a$-numerical index in the noncommutative case.
Indeed, one can show easily that for the matrix $A$ above we have
\begin{align*}
\mathfrak{A}_{A}=\left\{\begin{bmatrix}
\alpha & 0 \\
\beta & \gamma
\end{bmatrix}: \,(\alpha, \beta, \gamma)\in\mathbb{C}^3\right\}.
\end{align*}
We see that $\mathfrak{A}_{A}$ is non commutative, non self-adjoint and
the only nilpotent matrix in $\mathfrak{A}_{A}$ is $Y$ which satisfy $AY=0$.
On the other hand we have $n_{A}(\mathfrak{A})=1$.
\end{remark}
%%%%%%%%%%%%%%%%%%%
%%%%%%%%%%%%%%%%%%%
\bibliographystyle{amsplain}

\end{document}